\documentclass[12pt]{article}
\usepackage[english]{babel}
\usepackage[T1]{fontenc}

\usepackage[T2A]{fontenc}
\usepackage[utf8]{inputenc}

\usepackage{latexsym}
\usepackage{amssymb}
\usepackage{stmaryrd}
\usepackage{amsmath}
\usepackage{paralist} 
\usepackage[dvips]{graphicx,epsfig}
\usepackage{color}

\usepackage{amsmath,amsthm,amssymb}
\usepackage{hyperref}
\usepackage{textcomp}

\newtheorem{lemma}{Lemma}
\newtheorem{proposition}{Proposition}
\newtheorem{corollary}{Corollary}

\newtheorem*{statement}{Statement}

\newtheorem{rem}{Remark}

\def\sign{\rm{sign}}

\def\leq{\leqslant}
\def\geq{\geqslant}

\date{}

\begin{document}

\title{Remarks on asymptotic independence
\footnotemark[0]\footnotetext[0]{%
\textit{MSC 2000 subject classification}. Primary 60F99, 60G07, secondary
60B10,\\ 60B99 .} \footnotemark[0] 
\footnotemark[0]\footnotetext[0]{ \textit{Corresponding
author:} Youri Davydov,  e-mail: youri.davydov@univ-lille.fr }}

\author{ Youri Davydov$^{\text{\small 1}}$  and   Svyatoslav Novikov
$^{\text{\small 2}}$ \\
{\small $^{\text{1}}$ St. Petersburg State University}\\
{\small and Universit\'e de Lille, Laboratoire Paul
Painlev\'e}\\
 {\small $^{\text{2}}$ Chebyshev Laboratory, St. Petersburg State University  
 }}

\maketitle

\begin{abstract}
    In this paper we introduce several natural definitions  of asymptotic independence  of two sequences  of random elements. We discuss their basic properties, some simple connections between them and connections with
properties of weak dependence. In particular, the case of tight sequences  is considered in detail. Finally, in order to clarify the relationships between different definitions, we provide some counter\-examples.
\end{abstract}

\section{Introduction}
Questions related to asymptotic independence (AI) appear in different problems of probability theory and its applications. Intuitively, AI corresponds to the vanishing of "dependency", and since this "dependency" can be characterized in a number of ways, there is also a large specter of possibilities for the definition of AI. A few of them, that appear of most interest to us, are considered in the present work.

Evidently, the study of asymptotic independence adjoins the well-developed theory of weak dependence. There is a large literature on basic properties of weak dependence conditions. See, e.g. the survey
\cite{brad} and references therein. One can also mention the books
\cite{ibr}, \cite{doukh}, \cite{rio}, \cite{bul}.

Note that the main efforts of works on weak dependence are focused on the study of conditions ensuring the asymptotic independence of the past $(\sigma-{\mathrm {algebra}}\;\; \sigma\{X_k, k\leq 0\})$ and the future
$(\sigma-{\mathrm {algebra}}\;\; \sigma\{X_k, k\geq n\},\;$\\$ n\rightarrow\infty)$
of a given process,
with subsequent applications to limit theorems. 

In contrast, we are concerned with the asymptotic
independence of individual values $X_n$ and $Y_n$ and we are interested in conditions expressed in terms of the proximity of their joint distribution to the product of the marginal ones.

The question of the asymptotic relationship between $X_n$ and $Y_n$ occurs in a variety of problems.
Here is a typical example.

In the fundamental work \cite{b-b} was considered a random graph (the so-called Radial Spanning Tree) 
associated with the configuration of a homogeneous Poisson point process in $\mathbb{R}^2$.
Among other things, it was proved that the asymptotic directions of semi-infinite branches of this graph completely fill the circle.
To test the hypothesis about the uniformity of the distribution of asymptotic directions, it is natural to consider random variables $X_n$ equal to the number of semi-infinite branches intersecting 
the arc $n\alpha,$  where $\alpha \subset S^1$ and $n\rightarrow \infty$.
It turns out that for disjoint arcs $\alpha_1,\;\alpha_2$
 the corresponding variables $X_n^{(1)},\; X_n^{(2)}$ will be asymptotically independent in the sense of condition {\bf{AI-4}}, defined below, and this fact plays an important role in the subsequent analysis.

The paper consists of six sections. The first section is the introduction. Section 2 includes some basic implications between introduced conditions {\bf{AI-0}} - {\bf{AI-4}}. In Section 3, we provide some sufficient conditions for {{AI}} to hold, paying special attention to the case when $P_{X_n}$ and $P_{Y_n}$ are tight. In Section 4, we discuss the stability of {\bf{AI-0}} - {\bf{AI-4}} under transformations. Section 5 contains important counterexamples which clarify the relationships between {\bf{AI-0}} - {\bf{AI-4}}. Finally, in Section 6, we present some open questions related to asymptotic independence. 

\subsection{Conditions of AI}

Let $(X_n)$ and $(Y_n)$ be two sequences of random elements of measurable spaces $(E_1,\mathcal{E}_1)$ and $(E_2,\mathcal{E}_2)$ defined on the probability space
 $(\Omega,{\cal F},\mathbb{P}).$ We consider the pair $(X_n,Y_n)$ as an element of $(E_1 \times E_2,\; \mathcal{E}_1 \times \mathcal{E}_2)$.

It is reasonable to define asymptotic independence (AI) as the merging of the distributions $P_{(X_n,Y_n)}$ and $P_{X_n}\times P_{Y_n}$.

 Let us  consider first the case when $E_1$ and $E_2$ are Polish (that is, complete separable metric) spaces with metrics $d_1$ and $d_2$ respectively. Suppose that $\mathcal{E}_1$ and $\mathcal{E}_2$ are Borel $\sigma$-algebras of $E_1$ and $E_2$. Consider the space $E_1 \times E_2$ endowed with the product topology. We can suppose that it is generated by one of the metrics $d((x_1,y_1),(x_2,y_2))=d_1(x_1,x_2)+d_2(y_1,y_2)$ or $r((x_1,y_1),(x_2,y_2))=\max \{d_1(x_1,x_2),d_2(y_1,y_2)\}$. These metrics are equivalent: $r \leq d \leq 2r$.
 
In this case we can use the notions of weak convergence of measures and merging of measures. We say that two sequences $(\mu_n),\;(\nu_n)$ of probability measures defined on the same probability space
are {\bf merging} if 
$$
\pi(\mu_n, \;\nu_n) \;\rightarrow 0,\;\;\;n\rightarrow \infty,
$$
where $\pi$ is the \textit{L\'{e}vy-Prokhorov metric}
\begin{equation}
\pi (\mu,\nu)=\inf \{\varepsilon :\mu(A^{\varepsilon })\leq \nu(A)+\varepsilon \text{
\thinspace \thinspace for all closed sets }A\}.  \label{i3}
\end{equation}
(One can restrict themselves to only one inequality, without switching $\mu$ and $\nu$;
see, e.g  \cite{dudley}, Theorem 11.3.1.)

 Now considering asymptotic independence,  it is natural to introduce the following condition: $$(P_{(X_n,Y_n)})\; merges \; with \; (P_{X_n}\times P_{Y_n}) \; when \; n \to +\infty.$$ Due to \cite{D-R}, Theorem 1, it is equivalent to the following:

{\bf AI-1:} {\it For all bounded {\bf uniformly continuous} functions\\
 $h:E_1 \times E_2 \to \mathbb{R}, $ 
 $$
 \int h(x,y)P_{(X_n,Y_n)}(dx,dy)-\int h(x,y)(P_{X_n}\times P_{Y_n})(dx,dy) \to 0,
 $$ when $n \to +\infty$.}

We can also suggest a weaker condition:

{\bf AI-0:} {\it For all bounded {\bf uniformly continuous} functions\\
 $f:E_1 \to \mathbb{R},\;\,g:E_2 \to \mathbb{R},$ 
 $$
 \mathbb{E}f(X_n)g(Y_n)-\mathbb{E}f(X_n)\mathbb{E}g(Y_n) \to 0,
 $$ 
 when $n \to +\infty$.}
 
It is clear that \textbf{AI-1} implies \textbf{AI-0}.
\vspace{7pt}

In the general case several additional definitions can be suggested:
\vspace{5pt}

{\bf AI-2:}  For all
 $A \in \mathcal{E}_1$, $B \in \mathcal{E}_2,$  
 $$
 |P_{(X_n,Y_n)}(A \times B)-P_{X_n}(A)P_{Y_n}(B)| \to 0,\;\; n \to +\infty.
 $$

{\bf AI-3:}
$
\sup_{A \in \mathcal{E}_1,B\in \mathcal{E}_2}|P_{(X_n,Y_n)}(A \times B)-P_{X_n}(A)P_{Y_n}(B)| \to 0,\;\;\;n \to +\infty.
$
\vspace{10pt}

{\bf AI-4:}
$
\;\;\;\;||P_{(X_n,Y_n)}-P_{X_n}\times P_{Y_n}||_{var} \to 0,\;\;\;n \to +\infty.
$
\vspace{5pt}
 
Here $|| \cdot ||_{var}$ is the total variation norm.
\vspace{5pt}
 
Obviously, \textbf{AI-4}$\Rightarrow$\textbf{AI-3}$\Rightarrow$\textbf{AI-2}.
\vspace{7pt}

\vspace{7pt}

In theory, it would be possible to follow a more general approach by considering the AI of sequences of $\sigma$-algebras. For instance, here are the analogs of \textbf{AI-3} and \textbf{AI-4}.


 Let   
 $({\cal M}_n),\; ({\cal L}_n)$ be two sequences of sub-$\sigma$-algebras
 of the main probability space. Consider two conditions:
\vspace{5pt}
   
 I. (Analog of \textbf{AI-3}).
 $$
  \alpha({\cal M}_n, {\cal L}_n):=
  \sup_{{A \in {\cal M}_n, B \in {\cal L}_n}}
  |\mathbb{P}(A \cap B)-\mathbb{P}(A)\mathbb{P}(B)| \to 0, \;\; n \to +\infty.
  $$
  Note that in a similar way  the  $\alpha$-mixing coefficient is defined for any sub-$\sigma$-algebras of ${\cal F}$ as was proposed by M. Rosenblatt (\cite{ros}). 
\vspace{5pt}
  
  II. (Analog of \textbf{AI-4})
 $$
 \beta({\cal M}_n, {\cal L}_n):=
 \sup \frac{1}{2}\sum_{i=1}^I\sum_{j=1}^J
 |\mathbb{P}(A_i \cap B_j)-\mathbb{P}(A_i)\mathbb{P}(B_j)|
 \to 0, \;\; n \to +\infty,
 $$ 
   where the supremum is taken over all pairs of (finite) partitions 
   $\{A_1,\ldots ,A_I\}$
and $\{B_1,\ldots ,B_J\}$ of $\Omega$
 such that $A_i\in {\cal M}_n $ for each $i$ and 
 $B_j \in {\cal L}_n$ for each $j.$

   However, since in all cases known to us, $\sigma$-algebras are generated by concrete random elements, we only consider the case of AI of random elements.
\vspace{7pt}

There are many relations between conditions of \textbf{AI} and properties of weak dependence; we will now give a few examples.

Let $Z =(\xi_k)_{k\in \mathbb{Z}}$ be a strictly stationary sequence. 

{\bf a)} Let $Z_n$ be the shifted sequence: $(Z_n)_k=\xi_{n+k}$. Consider $Z$, $Z_n$ as random elements of the space $(E,\mathcal{E})$, where $E = \mathbb{R}^{\mathbb{Z}}$, $\mathcal{E}$ is the $\sigma$-algebra generated by cylindrical subsets of $E$. The distribution of $Z$ (denoted as $\mu$) is invariant with respect to the Bernoulli shift $T$ on $E$: $(T(\{a_n\}))_k=a_{k+1}$. 

The \textbf{mixing} condition (in the sense of ergodic theory) means that 
for all $A,B \in \mathcal{E}\:$
$$
    \mu(T^{-n}A \cap B)\to \mu(A)\mu(B),
$$
 when $n \to +\infty$. It is not difficult to see that it is equivalent to
 the following:
$$
     \mathbb{P}\{Z \in T^{-n}(A),\; Z \in B\} \to \mathbb{P}\{Z \in A\}\mathbb{P}\{Z \in B\},
$$
  that is, 
  $$
  \mathbb{P}\{Z_n \in A,Z \in B\} \to \mathbb{P}\{Z \in A\}\mathbb{P}\{Z \in B\}.
  $$ 
  Taking $X_n=Z_n$, $Y_n=Z$, we get 
  $$
  \mathbb{P}\{X_n \in A, Y_n \in B\}-\mathbb{P}\{X_n \in A\}\mathbb{P}\{Y_n \in B\} \to 0,
  $$
   when $n \to +\infty$ (as $Z$ is stationary, we have  
   $\mathbb{P}\{Z_n \in A\}=\mathbb{P}\{Z \in A\}$). This way we can see that \textbf{mixing} corresponds to \textbf{AI-2}.
\vspace{5pt}
  
{\bf b)} Let $X_n:\Omega \to \mathbb{R}^\mathbb{-\mathbb{N}}$ be a restriction of $Z$ to $\{...,-1,0\}$ and \\$Y_n:\Omega \to \mathbb{R}^{\mathbb{N}},$ be a restriction of $Z$ to $\{n,n + 1,...\}.$ Let $\mathcal{M}^b_a=\sigma\{\xi_a,...,\xi_b\}$.
 
 The \textbf{strong mixing} (or $\alpha$-mixing) condition introduced by Rosenblatt means that: 
\begin{equation}\label{alpha}
 \sup_{{A \in \mathcal{M}_{-\infty}^{0},\; B \in \mathcal{M}_n^{\infty}}} 
 |\mathbb{P}(A \cap B)-\mathbb{P}(A)\mathbb{P}(B)| \to 0,
 \end{equation}
  when $n \to +\infty$. 
  Let $\mathcal{E}_1$ be the $\sigma$-algebra generated by cylindrical subsets of $\mathbb{R}^{-\mathbb{N}}$ and let $\mathcal{E}_2$ be the $\sigma$-algebra generated by cylindrical subsets of $\mathbb{R}^{\mathbb{N}}$. Then (\ref{alpha}) is equivalent to 
  $$
  \sup_{{A \in \mathcal{E}_1,\; B \in \mathcal{E}_2}} 
  |P_{(X_n,\,Y_n)}(A \times B)-P_{X_n}(A)\times P_{Y_n}(B)| \to 0,\: n\to +\infty. 
  $$
  Hence, the \textbf{strong mixing} for stationary sequences coincides
 with  
  \textbf{AI-3}.
\vspace{5pt}
   
{\bf c)} Preserving the notation introduced in the first part of b)\\
  ($X_n=(...,\xi_{-1},\xi_0);\;Y_n=(\xi_n,\xi_{n+1},...)$), we get that the condition of 
  \textbf{complete regularity} introduced by Kolmogorov (so-called $\beta$-\textbf{mixing})
  $$ 
  ||P_{(X_n,Y_n)}-P_{X_n}\times P_{Y_n}||_{var} \to 0,\;\; n \to +\infty,
  $$
   coincides with \textbf{AI-4}.
\vspace{5pt}
   
 Because it is well known (see \cite{bradley}, 2.1, p. 112) that \textbf{mixing} does not imply \textbf{strong mixing} and the latter does not imply $\beta$-\textbf{mixing}, the same holds for \textbf{AI-2}, \textbf{AI-3}, \textbf{AI-4}.
\vspace{5pt}

\vspace{5pt}
   
\section{Connections between  AI-0 - AI-4}

Obviously, the condition {\bf{AI-4}} implies {\bf{AI-1}}.


\begin{proposition}
The condition  {\bf{AI-2}} implies {\bf{AI-0}}.
\end{proposition}
\begin{proof} The condition {\bf{AI-2}} means that {\bf{AI-0}} holds for indicator functions $f$ and $g.$ We finish the proof approximating
uniformly two given  bounded uniformly continuous functions.
\end{proof}
Later we will show that \textbf{AI-3} does not imply \textbf{AI-1}. In particular, \textbf{AI-0} does not imply \textbf{AI-1}.

Condition \textbf{AI-3} is equivalent  to the following formally weaker condition: 
$$ 
\sup_{A \in \mathcal{A}_1,\;B\in \mathcal{A}_2}|P_{(X_n,Y_n)}(A \times B)-P_{X_n}(A)P_{Y_n}(B)| \to 0,\; n \to +\infty, 
$$
 where $\mathcal{A}_1$ and $\mathcal{A}_2$ are some algebras which generate $\sigma$-algebras $\mathcal{E}_1$ and $\mathcal{E}_2$.

Let's mention another useful fact.

Consider the set 
$$
\Pi=\{\nu\;|\; \nu=\mu_1\times \mu_2,\;\;\mu_1,\mu_2\;
\mathrm {are\; probability\; measures\; on} \; \mathcal{E}_1\;\mathrm {and}\;\;\mathcal{E}_2\}.
$$
and denote $\pi(\mu,\Pi)$ the $\pi$-distance between measure $\mu$ and the set $\Pi,$
$$
\pi(\mu,\Pi) = \inf_{\nu\in \Pi}\{\pi(\mu,\nu)\}.
$$
\begin{proposition}
 The condition  {\bf{AI-1}} is equivalent to 
 \begin{equation}\label{lem2'''}
 \pi(P_{(X_n,Y_n)},\Pi)\to 0,\;\;\;n\to+\infty.
 \end{equation}
\end{proposition}
\begin{proof}
\textbf{AI-1}$\Rightarrow$(\ref{lem2'''}) is obvious. From (\ref{lem2'''}) to \textbf{AI-1}: in this case there exist $\;(\mu_n),(\nu_n)$ such that 
\begin{equation}\label{lem2''''}
\pi(P_{(X_n,Y_n)},\;\mu_n\times\nu_n) \to 0,\; n\to+\infty.
 \end{equation}
From (\ref{lem2''''}) and \cite{D-R}, Theorem 2,C, it follows that $\pi(P_{X_n},\: \mu_n) \to 0$ and $\pi(P_{Y_n},\: \nu_n) \to 0$.
Applying the Skorokhod embedding theorem \cite{D-R}, Theorem 2,B, we get that there exist two probability spaces $(\Omega_1,{\cal F}_1,\mathbb{P}_1)$ and $(\Omega_2,{\cal F}_2,\mathbb{P}_2)$, random elements $X_n',Y_n'$ on $\Omega_1$ and random elements $X_n'',Y_n''$ on $\Omega_2$ such that $P_{X_n'}=P_{X_n},\: P_{Y_n'}=\mu_n,\: P_{X_n''}=P_{Y_n},\: P_{Y_n''}=\nu_n$ and $d_1(X_n',Y_n')\xrightarrow{P} 0,\:d_2(X_n'',Y_n'')\xrightarrow{P} 0$. 

We can also consider $X_n',X_n'',Y_n',Y_n''$ as random elements of the probability space $\Omega_3=\Omega_1 \times \Omega_2$. Then $d((X_n',X_n''),\:(Y_n',Y_n''))\xrightarrow{P} 0$ and\\ $P_{(X_n',X_n'')}=P_{X_n}\times P_{Y_n}, \; P_{(Y_n',Y_n'')}=\mu_n \times \nu_n$. Hence, $$\pi(P_{X_n}\times P_{Y_n},\;\mu_n \times \nu_n) \to 0, \; n\to+\infty.$$ To conclude the proof, combine this with (\ref{lem2''''}). 
\end{proof} 
 \section{Sufficient conditions for AI} 
 Here we propose two useful sufficient conditions for verification of
 {\bf AI-4} and {\bf AI-3}.
 \begin{proposition}
 Suppose that $(X_n'),(Y_n')$ are such that:
 
 1) $X_n'$ and $Y_n'$ are independent for all $n;$
 
 2) $\mathbb{P}\{X_n \neq X_n'\} \to 0,\;\; \mathbb{P}\{Y_n \neq Y_n'\}\to 0,\;\; n \to +\infty$.
 
 Then for $(X_n)$ and $(Y_n)$ {\bf{AI-4}} holds.
 \end{proposition}
 \begin{proof}
 Remark that 2) implies $\mathbb{P}\{(X_n,Y_n)\neq(X_n', Y_n')\}\to 0, \; n\to+\infty$. To complete the proof, recall the well-known fact that if 
 $\mathbb{P}\{\xi \neq \eta\}=\delta$ then $||P_{\xi}-P_{\eta}||_{var} \leq 2\delta$.
 \end{proof}
 \begin{proposition}
 Suppose that $X_n$ and $Y_n$ are conditionally independent given $\Omega_n$ and $\mathbb{P}(\Omega_n)\to 1$. 
 
 Then {\bf{AI-3}} is satisfied. 
 \end{proposition}
 \begin{proof}
 Let $A \in \mathcal{E}_1, B \in \mathcal{E}_2$. It is clear that 
 $$
 \mathbb{P}\{X_n \in A,Y_n\in B\}=\mathbb{P}\{X_n \in A,Y_n \in B,\; \Omega_n\}+
 r_n,
 $$
  where $r_n\leq \delta_n:=P\{\Omega_n^c\}$. 
  It  follows from the conditional independence that 
  $$
  \mathbb{P}\{X_n \in A, Y_n \in B\;|\; \Omega_n\}=
 \mathbb{ P}\{X_n \in A\; |\; \Omega_n\}\mathbb{ P}\{Y_n \in B\; |\; \Omega_n\},
 $$
 that is, 
 $$
 \mathbb{P}\{X_n \in A, Y_n \in B,\; \Omega_n\}=
 \frac{1}{\mathbb{P}(\Omega_n)}\mathbb{P}\{X_n \in A ,\; \Omega_n\} 
 \mathbb{P}\{Y_n \in B ,\; \Omega_n\}.
 $$
  Hence 
\begin{align*}
\Delta_n:&=
 |\mathbb {P}\{X_n\in A,\; Y_n\in B\}-\mathbb {P}\{X_n \in A\}\mathbb {P}\{Y_n \in B\}|\\
&  \leq 
 \delta_n + |\mathbb {P}\{X_n \in A,\;Y_n \in B,\;\Omega_n\}-\mathbb {P}\{X_n\in A\}\mathbb {P}\{Y_n \in B\}|\\
 & =
 \delta_n+\left|\frac{1}{\mathbb {P}(\Omega_n)}\mathbb {P}\{X_n \in A , \Omega_n\} \mathbb {P}\{Y_n \in B , \Omega_n\}-\mathbb {P}\{X_n\in A\}\mathbb {P}\{Y_n \in B\}\right|.
 \end{align*}
 Therefore,
 \begin{align*} 
 \Delta_n&
 \leq \delta_n+\mathbb {P}\{X_n \in A, \Omega_n\}\left|\frac{1}{\mathbb {P}(\Omega_n)}\mathbb {P}\{Y_n\in B, \Omega_n\}-\mathbb {P}\{Y_n\in B\}\right|\\
& + \mathbb {P}\{Y_n \in B\}|\mathbb {P}\{X_n\in A, \Omega_n\}-\mathbb {P}\{X_n \in A\}|\\
&\leq 2\delta_n+\frac{1}{1-\delta_n}|\mathbb {P}\{Y_n\in B,\Omega_n\}-\mathbb {P}\{Y_n \in B\}\mathbb {P}(\Omega_n)|\\
&\leq 2\delta_n\left(1+\frac{1}{1-\delta_n}\right).
 \end{align*} 
  As this estimate is uniform in $A$ and in $B$, we get
   \textbf{AI-3}.
 \end{proof}

\subsection{Tight sequences}
Now we consider another important case when both sequences of distributions of  $X_n$ and $Y_n$ are tight.
\begin{proposition}
Suppose that $(P_{X_n})$ and $(P_{Y_n})$ are tight.
The following implications take place: 
\begin{center}
{\bf{AI-4}} $\Rightarrow$ {\bf{AI-3}} $\Rightarrow$ {\bf{AI-2}}  $\Rightarrow$ {\bf{AI-1}} $\Rightarrow$ {\bf{AI-0}}.
\end{center}
Moreover, in this case  {\bf{AI-0}} $\Rightarrow$ {\bf{AI-1}}.
\end{proposition}
\begin{proof}
We know that the first line of implications, except for {\bf{AI-2}} $\Rightarrow$ {\bf{AI-1}} always takes place, but {\bf{AI-2}} implies {\bf{AI-0}}, so we only have to prove \textbf{AI-0} $\Rightarrow$ \textbf{AI-1}. We will need two lemmas:
\begin{lemma}
Let $\mu$, $\nu$ be  probability measures on a product of two Polish spaces \\
$(E_1\times E_2,\; \mathcal{E}_1\times \mathcal{E}_2)$. 
If $\mu(F\times G) = \nu(F\times G)$ for all closed sets $F\in \mathcal{E}_1,\; G \in \mathcal{E}_2,$ then  $\mu =\nu$.
\end{lemma}
\begin{proof}
Let $\mu_1,\;\mu_2$ be marginal distributions of $\mu.$
By regularity of $\mu_1,\;\mu_2$ for arbitrary $A\in \mathcal{E}_1,\;
B\in \mathcal{E}_2$ there exist two sequences of closed sets\\
$F_n\subset A, F_n \in \mathcal{E}_1,\;\; 
G_n\subset B, G_n\in \mathcal{E}_2, $  such that
$$
\mu_1\{A\setminus F_n\}\rightarrow 0,\;\;\; 
\mu_2\{B\setminus G_n\}\rightarrow 0.
$$
As
$$
(A\times B)\setminus (F_n\times G_n) \subset
[(A\setminus F_n)\times E_2]\cup[E_1\times (B\setminus G_n)],
$$
we have 
$$
\mu\{(A\times B)\setminus (F_n\times G_n)\}\rightarrow 0.
$$
From this remark and the condition of lemma it follows 
that $\mu =\nu$ on the algebra generated by cells, hence they coincide on $\mathcal{E}_1\times \mathcal{E}_2.$
\end{proof}
\begin{lemma}
Suppose $(P_n)$, $(Q_n)$, $(L_n)$ are sequences  of probability measures on $\mathcal{E}_1$, $\mathcal{E}_2$ and $\mathcal{E}_1 \times \mathcal{E}_2$ respectively. Suppose that $P_n \Rightarrow P$, $Q_n \Rightarrow Q$, $L_n \Rightarrow L$. Moreover, suppose that  
for all bounded uniformly continuous functions\\ $f:E_1 \to \mathbb{R}$, $g:E_2 \to \mathbb{R}$, 
\begin{equation}\label{lem2}
\int f(x)g(y) L_n(dx,dy) - \int f(x) P_n(dx) \int g(y) Q_n(dy) \to 0,\; n\to+\infty. 
\end{equation}
Then $L = P \times Q$. 
\begin{rem}
If $E_1=\mathbb{R}^m,E_2=\mathbb{R}^k$, instead of (\ref{lem2}) the following condition on characteristic functions is sufficient: 

for all $\overline{t} \in E_1,\; \overline{s} \in E_2,$
\begin{equation}\label{lem2'}
\phi_{L_n}(\overline{t},\overline{s})-\phi_{P_n}(\overline{t})\phi_{Q_n}(\overline{s})\to 0,\;\;\; n\to +\infty.
\end{equation}
\end{rem}
\end{lemma}
\begin{proof} We omit the proof which is standard.
\end{proof}

Let us return to the proof of the implication \textbf{AI-0}$\Rightarrow$\textbf{AI-1}. Suppose it is not true. Then there exist $\delta>0$ and a subsequence $(n')\subset \mathbb{N}$ such that for all $n'$
\begin{equation}\label{lem2(3)}
\rho(P_{(X_{n'},Y_{n'})},P_{X_{n'}}\times P_{Y_{n'}})\geq \delta.
\end{equation} 
From $(n')$ choose $(n'')\subset (n')$ such that 
\begin{equation}\label{lem2(4)}
P_{X_{n''}}\Rightarrow P,\;P_{Y_{n''}}\Rightarrow Q, \;P_{(X_{n''},Y_{n''})}\Rightarrow L.
\end{equation} 
Due to condition \textbf{AI-0} and Lemma 2, it follows from (\ref{lem2(4)}) that $L = P \times Q$, which contradicts (\ref{lem2(3)}).
\end{proof}

\subsection{Case $E_1=\mathbb{R}^m$,\;$E_2=\mathbb{R}^k$.}
\begin{proposition}\label{char}
Suppose that $(P_{X_n})$ and $(P_{Y_n})$ are tight. The following conditions are equivalent:

1) {\bf{AI-1}}.

 2) For all  $ \overline{t} \in E_1, \overline{s} \in E_2,$ for
 characteristic functions 
 \begin{equation}\label{char-f}
 \phi_{(X_n,Y_n)}(\overline{t},\overline{s})-\phi_{X_n}(\overline{t})\phi_{Y_n}(\overline{s})\to 0,\;\;\; n\to +\infty.
 \end{equation}
\end{proposition}
\begin{proof}
Condition 2) can be rewritten as 
$$
\mathbb{E}e^{i\overline{t}X_n}e^{i\overline{s}Y_n}-\mathbb{E}e^{i\overline{t}\cdot X_n}\mathbb{E}e^{i\overline{s}\cdot Y_n}\to 0,\;\;\; n\to +\infty.
$$ 
This condition follows from \textbf{AI-0}, hence, 1)$\Rightarrow$2).

Now suppose that 2) holds but 1) does not hold. Then there exist $\delta>0$ and a subsequence $(n')\subset \mathbb{N}$ such that for all $n'$
\begin{equation}\label{prop7}
\pi(P_{(X_{n'},Y_{n'})},P_{X_{n'}}\times P_{Y_{n'}})\geq \delta.
\end{equation}
Due to relative compactness we can find a subsequence $(n'')\subset(n')$ such that
$$
P_{X_{n''}}\Rightarrow P,\;\;P_{Y_{n''}}\Rightarrow Q,\;\;P_{(X_{n''},Y_{n''})}\Rightarrow L.
$$ 
From 2) and Lemma 2 we have $L=P\times Q$. This is a contradiction with (\ref{prop7}).
\end{proof}

Consider now the case when the joint distribution $P_{(X_n,\,Y_n)}$ is Gaussian. Let $X_n = (X_n^{(1)},\ldots,X_n^{(m)}),\;\;
Y_n = (Y_n^{(1)},\ldots,Y_n^{(k)}),\;\,EX_n=a_n,\;$\\$EY_n = b_n,$ and
$ {\text {cov}} \{X_n^{(i)},\,Y_n^{(j)}\} = r_n^{i,j}, \;1\leq i\leq m,\;1\leq j\leq k.$
\begin{proposition}
The following conditions are equivalent:

1) The sequences $(P_{X_n}),\;(P_{Y_n})$ are tight and satisfy the condition
{\bf{AI-1}}.

 2) The sequences $(a_n), \,(b_n),\, (E|X_n|^2),\;(E|Y_n|^2)$ are bounded and for all $i,j,\;1\leq i\leq m,\;1\leq j\leq k,$
 \begin{equation}\label{corr}
  r_n^{i,j} \rightarrow 0.
  \end{equation}
\end{proposition}
\begin{proof}
It is well known that for Gaussian vectors boundedness of the first two moments is equivalent to tightness.

It is clear that the condition  (\ref{corr}) gives (\ref{char-f}). 
 Using tightness it is easy to check that (\ref{corr}) also follows from
 (\ref{char-f}), hence
due to Prop.\ref{char}
 the equivalence $1) \Longleftrightarrow 2)$ follows.
\end{proof}

\section{AI under transformations}
If $(X_n),\,(Y_n)$ are independent sequences of random elements, 
then their images $(f_n(X_n)),\,(g_n(Y_n))$ under arbitrary sequences of measurable mappings $(f_n),\,(g_n)$ are also independent.
This stability property does not hold in ge\-ne\-ral when we change independence to 
asymptotic independence.
   
   The next proposition contains information about stability of different types
   of AI under different classes of transformations.

\begin{proposition}
${}$

\begin{itemize}
\item[1)]
 If $(X_n),\,(Y_n)$ satisfy {\bf{AI-0}} then $(u(X_n)),\,(v(Y_n))$ satisfy {\bf{AI-0}} for all 
 uniformly continuous functions $u,v$.

\item[2)] If $(X_n),\,(Y_n)$ satisfy {\bf{AI-1}} then $(u(X_n)),\,(v(Y_n))$ satisfy {\bf{AI-1}} for all 
uniformly continuous functions $u,v$.

\item[3)] If $(X_n),\,(Y_n)$ satisfy {\bf{AI-2}} then $(u(X_n)),\,(v(Y_n))$ satisfy {\bf{AI-2}} for all measurable functions $u,v$.

\item[4)] If $(X_n),\,(Y_n)$ satisfy {\bf{AI-3}} then $(u_n(X_n)),\,(v_n(Y_n))$ satisfy {\bf{AI-3}} for all measurable functions $u_n,v_n$.

\item[5)] If $(X_n),\,(Y_n)$ satisfy {\bf{AI-4}} then $(u_n(X_n)),\,(v_n(Y_n))$ satisfy {\bf{AI-4}} for all measurable functions $u_n,v_n$.
\end{itemize}
\end{proposition}
\begin{proof}
Properties $1)\,-\,4)$ follow directly from initial hypotheses.


For 5) consider the mapping $w_n:(x,y)\mapsto (u_n(x),v_n(y))$. Then
due to \textbf{AI-4} for $(X_n)$ and $(Y_n)$ we have
\begin{align*}
||P_{(u_n(X_n),\,v_n(Y_n))}&-P_{u_n(X_n)}\times P_{v_n(Y_n)}||_{var}\nonumber \\
&= \displaystyle{||(P_{(X_n,Y_n)}-P_{X_n}\times P_{Y_n})w_n^{-1}||_{var}}\nonumber \\
&\leq \displaystyle{||P_{(X_n,Y_n)}-P_{X_n}\times P_{Y_n}||_{var} \to 0,}
\end{align*}
when $\;n\to+\infty.$
\end{proof}
\section{Counterexamples}
In this section we would like  to clarify the relationships between \textbf{AI-0} - \textbf{AI-4}. In order to do so, we will provide some counterexamples.

We saw in section 3.1 that in the case of tight sequences  the conditions
{\bf AI-1} and {\bf AI-0} are equivalent. It is easy to see that without the tightness assumption this equivalence
 will be preserved if the metric spaces 
$E_1$ and $E_2$ are compact. 
Indeed, let $U_1,\;U_2$ and $U$ be the spaces of 
bounded and uniformly continuous
real functions defined respectively 
on $E_1,\;E_2$ and\\ $E_1 \times E_2.$
Let $H$ be the closed subspace of $U$ formed by 
all linear combinations of the 
form $\sum_1^n f_i(x)g_i(y),$ where $f_i \in U_1,\; g_i\in U_2,\;n\in {\mathbb{N}}.$ 
If $E_1$ and $E_2$ are compact, it follows 
from the Stone-Weierstrass theorem (\cite{dunford}, VI.6.Th.16) that $H$ coincides with $U$, and it allows to
easily deduce {\bf AI-1} from {\bf AI-0}. 

The example constructed below shows that even in locally compact spaces
the equivalence between {\bf AI-1} and {\bf AI-0} may fail.

\subsection{AI-3 does not imply AI-1}

\begin{proposition}\label{contr1}
Let $E_1=\mathbb{R}$, $E_2=\mathbb{R}$. There exist sequences $(X_n)$ and $(Y_n)$
of random variables satisfying {\bf AI-3} but not {\bf AI-1}.
\end{proposition}
As {\bf AI-0} always follows from  {\bf AI-3} we immediately deduce 
\begin{corollary}\label{contr2}
Property {\bf AI-0} does not imply {\bf AI-1}.
\end{corollary}
As a byproduct we get also
\begin{corollary}\label{contr3}
For non compact spaces  $E_1,\;E_2$ (even if they are locally compact), it is possible that the
equality $U_1\times U_2 = U$ does not hold.
\end{corollary}
\begin{proof} Denote as $B_1(E_1), B_1(E_2)$ the sets of bounded real-valued measurable functions the absolute values of which do not exceed $1$ from $E_1$ and $E_2$ respectively.
It is easy to see that {\bf AI-3} is equivalent to the following property
\begin{equation}\label{AI-3a}
\sup_{f\in B_1(E_1),\;g\in B_1(E_2)}|\mathbb{E}f(X_n)g(Y_n)-\mathbb{E}f(X_n)\mathbb{E}g(Y_n)| \to 0,
\end{equation}
 when $n \to +\infty$.
It is sufficient to show that (\ref{AI-3a}) does not imply {\bf AI-1}.

Each $j \in \mathbb{N}$ admits a binary coding, i.e. $j=l_n...l_0$, where
$l_k \in \{0,1\},\,k=0,...,n,\,l_n=1,\,n\in \mathbb{N}.$ In other words, $j=1\cdot 2^n+l_{n-1}\cdot 2^{n-1}+...+l_0\cdot 2^0$.
Introduce the function 
\begin{equation*}
\chi(i,j) =
\begin{cases}
l_i & \text{ for } i \leq n\\
0 & \text{ for } i > n
\end{cases} 
\end{equation*}
(The function $\chi$ is defined on pairs of non negative integers). Also introduce the function $\sign(i,j)=2\chi(i,j)-1$.

Let $u:\mathbb{R}^2\to\mathbb{R}$, $u(x,y)=\max\{0,\;1-4|x|-4|y|\}$.
Remark that its support lies in the square $|x|\leq \frac{1}{4},|y|\leq \frac{1}{4}.$
Consider the function 
$$
h(x,y)=\sum\limits_{i=0}^{\infty}\sum\limits_{j=0}^{\infty}\chi(i,j)u(x-i,y-j).
$$ It is easy to see that $h$ is well-defined (when $x$, $y$ are fixed, no more than one of $u(x-i,y-j)$ is not equal to 0), bounded and uniformly continuous. 

By $\delta_{a}$ we will denote a delta-measure concentrated at the point $a \in \mathbb{R}^2$: $\delta_{a}(A)=1$ if $a \in A$, $\delta_{a}(A)=0$ if $a \notin A$. Now remark that there exist $X_n,Y_n$ such that 
$$
n 2^n P_{(X_n,Y_n)}=\sum\limits_{i=0}^{n-1}\sum\limits_{j=0}^{2^n-1}\chi(i,j)\delta_{(i,j)}+\sum\limits_{i=n}^{2n-1}\sum\limits_{j=0}^{2^n-1}(1-\chi(i-n,j))\delta_{(i,j)}.
$$

Indeed, $ \forall i: 0\leq i \leq n-1,$ 
\begin{equation}\label{*}
 \sum\limits_{j=0}^{2^n-1}\chi(i,j)=\sum\limits_{j=0}^{2^n-1}(1-\chi(i,j))=2^{n-1},
\end{equation} 
and then the total variation of the measure 
 $$
 \nu=\sum\limits_{i=0}^{n-1}\sum\limits_{j=0}^{2^n-1}\chi(i,j)\delta_{(i,j)}+\sum\limits_{i=n}^{2n-1}\sum\limits_{j=0}^{2^n-1}(1-\chi(i-n,j))\delta_{(i,j)} 
$$
equals $2^{n-1}\cdot 2n=n2^n$.

 Below an illustration for $n=3$ is given (the number in the $i$-th column and in the $j$-th row from below is the weight of the measure $P_{(X_n,Y_n)}$ at the point $(i,j)$).

\[
\left(%
\begin{array}{cccccc}
1/24 & 1/24 & 1/24 & 0 & 0 & 0\\
0 & 1/24 & 1/24 & 1/24 & 0 & 0\\
1/24 & 0 & 1/24 & 0 & 1/24 & 0\\
0 & 0 & 1/24 & 1/24 & 1/24 & 0\\
1/24 & 1/24 & 0 & 0 & 0 & 1/24\\
0 & 1/24 & 0 & 1/24 & 0 & 1/24\\
1/24 & 0 & 0 & 0 & 1/24 & 1/24\\
0 & 0 & 0 & 1/24 & 1/24 & 1/24\\
\end{array}%
\right)
\]

 It is easy to see from (\ref{*}) that for all $i,\;0\leq i \leq 2n-1,\;$\\ $P_{X_n}(\{i\})=\frac{1}{n2^n}\cdot 2 ^{n-1}=\frac{1}{2n}$. As
 for all $i,j,\;0\leq i \leq n-1,\,\;0 \leq j \leq 2^n-1,$ 
 $$
 P_{(X_n,Y_n)}\{(i,j)\}+P_{(X_n,Y_n)}\{(i+n,j)\}=\frac{\chi(i,j)+(1-\chi(i,j))}{n2^n}=\frac{1}{n2^n},
 $$ 
 we have $P_{Y_n}(\{j\})=\frac{1}{n2^n}\cdot n = \frac{1}{2^n}$.
  
 Consider $\mu_n = P_{(X_n,Y_n)}-P_{X_n}\times P_{Y_n}$, then
 \begin{equation*}
   \mu_n\{(i,j)\} = \frac{\chi(i,j)}{n2^n}-\frac{1}{2n}\cdot \frac{1}{2^n}=\frac{\sign(i,j)}{n2^{n+1}}, 
 \end{equation*}
when $0 \leq i \leq n-1,\; 0 \leq j \leq 2^n-1,$

and
\begin{equation*}
  \mu_n\{(i,j)\} = \frac{1-\chi(i-n,j)}{n2^n}-\frac{1}{2n}\cdot \frac{1}{2^n}=\frac{-\sign(i-n,j)}{n2^{n+1}},
 \end{equation*}
when $ n \leq i \leq 2n-1,\; 0 \leq j \leq 2^n-1.$

 Remark that 
 $$
 \int h(x,y) d\mu_n(x,y) = \sum\limits_{i=0}^{2n-1}\sum\limits_{j=0}^{2^n-1} h(i,j)\mu_n\{(i,j)\}=\sum\limits_{i=0}^{n-1}\sum\limits_{j=0}^{2^n-1} h(i,j)\mu_n\{(i,j)\},
 $$ 
 because when $i \geq n,\; j \leq 2^n-1,\;$ we have $h(i,j)=0$.
 
 We also have
 $$
 \sum\limits_{i=0}^{n-1}\sum\limits_{j=0}^{2^n-1} h(i,j)\mu_n\{(i,j)\}=\sum\limits_{i=0}^{n-1}\sum\limits_{j=0}^{2^n-1}
 \frac{\chi(i,j)\sign(i,j)}{n2^{n+1}}=\sum\limits_{i=0}^{n-1}\frac{2^{n-1}}{n2^{n+1}}=1/4
 $$
from (\ref{*}) and because when $\chi(i,j)=1$ we have $\sign(i,j)=1$. 
 
 Hence, $\int h(x,y) d\mu_n\{(x,y)\} \nrightarrow 0$ when $n \to +\infty$, and $(X_n,Y_n)$ do not satisfy condition \textbf{AI-1}.
 
 We will prove that $(X_n,Y_n)$ satisfy (\ref{AI-3a}). At first we will state the following:
 \begin{lemma}
 Suppose $a_0,...,a_{n-1},b_0,...,b_{2^n-1}$ are real numbers with absolute value not exceeding one. Then 
 $$
 \left|\sum\limits_{i=0}^{n-1}\sum\limits_{j=0}^{2^n-1}a_i b_j \sign(i,j)\right| \leq  2^n \sqrt{n}.
 $$
 \end{lemma}
 \begin{proof}
 Denote $\psi(a_0,...,a_{n-1},b_0,...,b_{2^n-1}) = \sum\limits_{i=0}^{n-1}\sum\limits_{j=0}^{2^n-1}a_i b_j \sign(i,j)$.
 If we consider $\psi$ as a function of one of the variables $a_0,...,a_{n-1},b_0,...,b_{2^n-1}$, fixing the remaining variables, we will get the sum of a constant and a linear function. Hence, $|\psi|$ is convex in  each variable. Then
 \begin{align} \label{**} 
 |\psi(...,a_{i-1},& \displaystyle{a_i,a_{i+1},...)|}  \nonumber \\
& \displaystyle{\leq \max\{|\psi(...,a_{i-1},1,a_{i+1},...)|,\;
|\psi(...,a_{i-1},-1,a_{i+1},...)|\},}\\ \nonumber 
 \end{align} 
and
  \begin{align} \label{***}
 |\psi(...,b_{j-1},& b_j,b_{j+1},...)| \nonumber\\
& \displaystyle{\leq \max\{|\psi(...,b_{j-1},1,b_{j+1},...)|,\;
|\psi(...,b_{j-1},-1,b_{j+1},...)|.}\\ \nonumber 
 \end{align}  
 Applying the inequalities (\ref{**}), (\ref{***}) consequentially for each variable, we get that $|\psi|$ reaches its maximum at some values of $a_0,...,a_{n-1},b_0,...,b_{2^n-1}$ such that for each $i$, $j$ $|a_i|=1$, $|b_j|=1$. 
 
 Remark that the number $\psi(a_0,...,a_{n-1},b_0,...,b_{2^n-1})$ can be constructed in the following way: consider the matrix $M$ of $n$ columns and $2^n$ rows, where $m_{ji}=\sign(i,j)$, then multiply the $i$-th column by $a_i$, multiply the $j$-th row by $b_j$ (first multiply the columns, second multiply the rows). The matrix $M$ will be transformed to $\tilde{M}$, then $\psi(a_0,...,a_{n-1},b_0,...,b_{2^n-1})$ is the sum of the elements of $\tilde{M}$. 
 
 The rows of $M$ are all possible strings of $n$ numbers 1 and (-1). It is shown above that we can suppose $a_i = 1$ or $a_i=-1$.
 After each multiplication by $a_i$ the matrix $M$ will still consist of numbers 1 and -1; distinct rows will remain distinct, so, after the multiplication by all $a_i$ we will get the matrix $\tilde{\tilde{M}}$, consisting of $2^n$ distinct rows of numbers 1 and -1. But the rows have length $n$, hence, each  string of length $n$ of numbers 1 and -1 will be represented exactly once. 
 
 Consider a row of the matrix $\tilde{\tilde{M}}$. Suppose there are $t$ numbers 1 in it. Then after the multiplication by $b_j$ the sum of the numbers in this row will not exceed $|t-(n-t)|$ in absolute value (we suppose that $\forall j$ $b_j=1$ or $b_j=-1$). Hence, 
\begin{equation}\label{****}
 |\psi(a_0,...,a_{n-1},b_0,...,b_{2^n-1})|\; \leq \sum\limits_{c_1,...,c_n \in \{-1,1 \} } |c_1+...+c_n|.
 \end{equation}
 Take $n$ i.i.d random variables $\epsilon_1,...,\epsilon_n$ such that\\
  $\mathbb{P}\{\epsilon_1 = 1\}=\mathbb{P}\{\epsilon_1=-1\}=1/2$. Then the right part of (\ref{****}) is equal to 
  $$
  2^n \mathbb{E}|\epsilon_1+...+\epsilon_n|\leq 2^n \sqrt{\mathbb{E}{|\epsilon_1+...+\epsilon_n|^2}}=2^n \sqrt{Var(\epsilon_1+...+\epsilon_n)}=2^n\sqrt{n},
  $$ 
  due to Jensen's inequality. 
 \end{proof}
 
 Now take any $f\in B_1(E_1)$, $g \in B_1(E_2)$ and remark that 
 \begin{align*}
 \int f(x)& g(y) d\mu_n(x,y)\\
 & =  \sum\limits_{i=0}^{n-1}\sum\limits_{j=0}^{2^n-1} \frac{\sign(i,j)}{n2^{n+1}} f(i)g(j) + \sum\limits_{i=n}^{2n-1}\sum\limits_{j=0}^{2^n-1} \frac{-\sign(i-n,j)}{n2^{n+1}} f(i)g(j).
 \end{align*}
  By lemma 1, substituting $a_i = f(i)$, $b_j=g(j)$, we have that the first sum does not exceed $\frac{2^n \sqrt{n}}{n2^{n+1}}=\frac{1}{2\sqrt{n}}$ in absolute value. By lemma 1, substituting $a_i = -f(i+n)$, $b_j = g(j)$, we have that the second sum also does not exceed $\frac{1}{2\sqrt{n}}$ in absolute value. 
 Finally, for all $n \in \mathbb{N},\, n\geq 1,$
 $$
 \int f(x)g(y) d\mu_n(x,y) \leq \frac{1}{\sqrt{n}} \to 0,
 $$ 
 when $n \to +\infty$. Hence, $(X_n,Y_n)$ satisfy condition (\ref{AI-3a}). 
 \end{proof}
\subsection{AI-1 does not imply AI-2}
\begin{proposition}\label{AI-1,2}
Let $E_1=\mathbb{R}$, $E_2=\mathbb{R}$. There exist sequences $(X_n)$ and $(Y_n)$
of random variables satisfying {\bf AI-1} but not {\bf AI-2}.
\end{proposition}
\begin{proof}
 Consider two independent random variables $X$, $Y$ such that $\mathbb{P}\{X=1\}=\mathbb{P}\{X=0\}=\frac{1}{2}$,\\ $\mathbb{P}\{Y=1\}=\mathbb{P}\{Y=0\}=\frac{1}{2}$.
Now suppose $X_n=X+\frac{Y}{n}$, $Y_n=Y$. 

Let us at first ensure that $X_n$, $Y_n$ do not satisfy condition {\bf AI-2}. Substitute $A=\{1\},B=\{1\}$. Then when $n \geq 2$ it is easy to check that $\mathbb{P}\{X_n \in A, Y_n \in B\}=0$, $\mathbb{P}\{X_n \in A\}\mathbb{P}\{Y_n \in B\}=\frac{1}{2}\cdot\frac{1}{4}=\frac{1}{8}$. So $\lim\limits_{n \to +\infty} |\mathbb{P}\{X_n \in A\}\mathbb{P}\{Y_n \in B\}-\mathbb{P}\{X_n \in A, Y_n \in B\}| \neq 0$, which is a contradiction with {\bf AI-2}.

We will check that, nevertheless, $X_n$, $Y_n$ satisfy {\bf AI-1}. Recall the following simple result:
\begin{statement}
Let $\{x_n\},\{y_n\}$ be two sequences of points in the metric space $(M,d)$, then for all $f:M\to \mathbb{R}^1$, uniformly continuous and bounded, $\int f d(\delta_{x_n})-\int f d(\delta_{y_n}) \to 0$ when $n \to +\infty$, if $d(x_n,y_n)\to 0$ when $n \to +\infty$. Here $\delta_{a}$ is a delta-measure concentrated at the point $a$.
\end{statement}

Remark that 
\begin{align*}
P_{(X_n,Y_n)}& =\delta_{(0,0)}\mathbb{P}\{X=0,Y=0\}+\delta_{(\frac{1}{n},1)}\mathbb{P}\{X=0,Y=1\}\\
&+\delta_{(1,0)}\mathbb{P}\{X=1,Y=0\}+\delta_{(1+\frac{1}{n},1)}\mathbb{P}\{X=1,Y=1\}\\
&=\frac{\delta_{(0,0)}+\delta_{(\frac{1}{n},1)}+\delta_{(1,0)}+\delta_{(1+\frac{1}{n},1)}}{4}. 
\end{align*}
Moreover, 
\begin{align*}
P_{X_n} \times P_{Y_n}& = \left(\frac{\delta_{0}+\delta_{\frac{1}{n}}+\delta_{1}+\delta_{1+\frac{1}{n}}}{4}\right)\times\left(\frac{\delta_{0}+\delta_{1}}{2}\right)\\
&=\frac{\delta_{(1,1)}+\delta_{(\frac{n+1}{n},1)}}{8}
+\frac{\delta_{(0,0)}+\delta_{(\frac{1}{n},0)}}{8}\\
&+\frac{\delta_{(0,1)}+\delta_{(\frac{1}{n},1)}}{8}
+\frac{\delta_{(\frac{n+1}{n},0)}+\delta_{(1,0)}}{8}.
\end{align*}
Finally for all bounded and uniformly continuous functions $h$ on $E_1 \times E_2$ we get 
\begin{align*}
\int h d(P_{(X_n,Y_n)}-& P_{X_n}\times P_{Y_n}) \\
& = \frac{1}{8} \int h\, d(\delta_{(\frac{n+1}{n},1)}-\delta_{(1,1)})
+\frac{1}{8} \int h\, d(\delta_{(0,0)}-\delta_{(\frac{1}{n},0)})\\
&+\frac{1}{8} \int h\, d(\delta_{(\frac{1}{n},1)}-\delta_{(0,1)})
+\frac{1}{8} \int h\, d(\delta_{(1,0)}-\delta_{(\frac{n+1}{n},0)}),
\end{align*}
 which goes to $0$ when $n \to +\infty$ (from the fact given above).
\end{proof}

\section{Concluding remarks}
{\bf I.}
It is clear that conditions {\bf AI-0 - AI-4} can naturally be modified for mutual asymptotic
independence of several (more than 2) random sequences.
At the same time analogs of all main given results
will remain true.
\vspace{5pt}

\noindent
{\bf II.}
Below we formulate some open questions and suggest some directions of research on asymptotic independence.

1. It is interesting to consider conditions for AI of the form
$$
\int_{E_1}fdP_{X_n} \int_{E_2}gdP_{Y_n} - \int_{E_1\times E_2}(f\times g)dP_{(X_n,Y_n)}
\rightarrow 0
$$
for all $f,\,g$ belonging to some classes ${\cal F}_1, \;{\cal F}_2$ of functions.
\vspace{5pt}

\noindent
2.  Find sufficient conditions for AI of the following type:

If $(f(X_n)),\;(g(Y_n))$ are AI for all $f,\,g$ belonging to some classes ${\cal F}_1, \;{\cal F}_2$ of functions, then $(X_n),\;(Y_n)$ are AI.
\vspace{5pt}

\noindent
3. One can also study conditions for AI of the form
$$
\int_{E_1}fdP_{X_n} \int_{E_2}gdP_{Y_n} - \int_{E_1\times E_2}(f\times g)dP_{(X_n,Y_n)}
\rightarrow 0
$$
\textbf{uniformly} for all $f,\,g$ belonging to some classes ${\cal F}_1, \;{\cal F}_2$ of functions. For example the following question is interesting:

For a metric space $M$ we denote by $BL_1(M)$ the set of real-valued functions $f$ on $M$ such that for all $x \in M, \; |f(x)|\leq 1,$ and for all $x,y \in M, \; |f(x)-f(y)|\leq |x-y|$.
Does {\bf{AI-0}} imply
$$
\sup_{f \in BL_1(E_1),\; g \in BL_1(E_2)} \left|\int\limits_{E_1\times E_2} f(x)g(y)dP_{(X_n,Y_n)}-\int\limits_{E_1} f(x)dP_{X_n}\int\limits_{E_2} g(y)dP_{Y_n}\right| \to 0,
$$
when $n \to +\infty?$

It is known that an analogous fact is true for {\bf{AI-1}}: {\bf{AI-1}} implies
$$
\sup_{h \in BL_1(E_1 \times E_2)} \left| \int\limits_{E_1\times E_2}h(x,y)dP_{(X_n,Y_n)}-\int\limits_{E_1 \times E_2}h(x,y)d(P_{X_n}\times P_{Y_n}) \right| \to 0,
$$
when $ n \to +\infty,$
look at \cite{D-R}, Corollary 6, for example.
\vspace{5pt}

\noindent
4. Does {\bf{AI-0}} imply {\bf{AI-1}} if only one of the sequences $P_{X_n}$ and $P_{Y_n}$ is tight?
\vspace{5pt}

\noindent
5. It is interesting to consider conditions for AI for random elements of 
concrete spaces (such as: space of sequences, space $C[0,1]$, space of configurations and so on...).




\vspace{15pt}

{\bf Acknowledgments}
\vspace{5pt}

1) The authors are grateful to V. Rotar' for useful discussions and
his interest to our work.

2) We would like to thank the anonymous reviewer whose detailed comments allowed us to
improve the presentation of the material.

3) Research is partially supported by «Native towns», a social investment program of PJSC «Gazprom Neft».

\end{document}